\documentclass[11pt,
]{amsart}
\usepackage{
amssymb, graphicx, color, cite
}
\usepackage[T1]{fontenc}

\date{\today}

\numberwithin{equation}{section}%

\newtheorem{theorem}{Theorem}[section]
\newtheorem{proposition}{Proposition}[section]

\newtheorem{definition}{Definition}[section]
\newtheorem{corollary}{Corollary}[section]

\theoremstyle{definition}
\newtheorem{remark}{Remark}[section]

\DeclareMathOperator{\grad}{grad}

\newcommand{\eps}{\varepsilon}

\newcommand{\R}{{\mathbb R}}

\newcommand{\Id}{\mbox{Id}}
\renewcommand{\r}[1]{(\ref{#1})}

\newcommand{\be}[1]{\begin{equation}\label{#1}}
\newcommand{\ee}{\end{equation}}

\renewcommand{\d}{\mathrm{d}}

\newcommand{\bo}{\partial M}

\newtheorem{example}{Example}

\title[The Lorentzian scattering rigidity problem]{The Lorentzian scattering rigidity problem\\
and rigidity of stationary metrics}

\author[P. Stefanov]{Plamen Stefanov}
\address{Department of Mathematics, Purdue University, West Lafayette, IN 47907}
\thanks{P.S. partly supported by  NSF  Grant DMS-2154489}

\begin{document}
\begin{abstract}
We study  scattering rigidity  in Lorentzian geometry: recovery of a Lorentz\-ian metric from the scattering relation $\mathcal S^\sharp$ known on a lateral boundary.  We show that, under a non-conjugacy assumption,  every defining function $r(x,y)$ of the submanifold of pairs of boundary points which can be connected by a lightlike geodesic plays the role of the boundary distance function in the Riemannian case in the following sense. Its linearization is the light ray transform of tensor fields of order two which are the perturbations of the metric. Next, we study scattering rigidity of stationary metrics in time-space cylinders and show that it can be reduced to boundary rigidity of magnetic systems on the base; a problem studied previously. This implies several scattering rigidity results for stationary metrics. 
\end{abstract} 
\maketitle

\section{Introduction} 
We study the scattering relation $\mathcal S^\sharp$ for Lorentzian manifolds with spacelike or timelike boundaries. The main question we are interested in is whether one can recover the metric $g$, up to some group of explicit gauge transformations, given $\mathcal S^\sharp$. Our convention is that $\mathcal S^\sharp$ acts on projections of lightlike  \textit{covectors} on the boundary; for its vector version we use the notation $\mathcal{S}$ but our main interest is in $\mathcal S^\sharp$.  While the question in this generality is still unanswered, and most likely requires assumptions beyond the ones we require for Riemannian metrics, we concentrate on the following three problems.

\smallskip
\textbf{(A)} Suppose we want to solve it by linearization near a background metric $g$. In fact, some of the strongest results in the Riemannian case are obtained this way. In the Riemannian case, a linearization of the boundary distance function $\rho$ leads to the inversion of the geodesic X-ray transform of tensor fields
\be{3}
Xf(x,v)= \int \left\langle f ,\dot\gamma_{x,v}(t) \otimes \dot\gamma_{x,v}(t)\right\rangle\d t, 
\ee
where the two-tensor $f$ is the perturbation of $g$. We want to show that $Xf=0$ implies that $f$ is potential (see next section), which linearizes the diffeomorphism invariance of the scattering rigidity problem. Moreover, we want to prove a stability estimate allowing us to treat the nonlinearity, see, e.g., \cite{SU-JFA09}. 

If we linearize $\mathcal S^\sharp$ instead, we get a not so simple looking formula. We still get a geodesic X-ray transform of symmetric tensor fields but of $\nabla f$, with a weight, plus a zeroth order term, see \cite{SU-MRL, SUV_localrigidity, SUV_anisotropic}. The appearance of derivatives is not surprising in view of \r{1} below or by the fact that the generator of the geodesic flow contains first derivatives of the metric. The simple looking \r{3} is more attractive for analysis however, especially in the Euclidean case where one can use the Fourier transform, see, e.g., \cite{Sh-book}. On the other hand, at least for simple manifolds, $\mathcal S^\sharp$ and $\rho$ determine each other but $\rho$ has the advantage of being scalar, and as we said, with a  simpler linearization. 

There is no obvious extension of the boundary distance function \textit{for this purpose} in the Lorentzian case even though distance/separation functions have been defined. In fact, the distance between two points that can be connected by a lightlike geodesic is zero, and if they cannot, one can define a separation function but that depends on more than the lightlike rays only.

We propose, for Lorentzian manifolds with spacelike or timelike boundaries,  an equivalent to the boundary distance function which can serve as a generating function for the scattering relation and which linearization about a fixed metric leads to the light ray X-ray transform of tensor fields of order two 
\be{LRT}
Lf(x,v)= \int \left\langle f ,\dot\gamma_{x,v}(t) \otimes \dot\gamma_{x,v}(t)\right\rangle\d t,
\ee
which looks formally like \r{3} but the background metric is a Lorentzian one, and we integrate over lightlike geodesics only. 
This defining function can be taken, in fact, to be a defining function of the boundary pairs $(x,y)$ which can be connected by a lightlike geodesic, and one way to construct it is through the energy functional. This legitimizes the interest in the Lorentzian version $L$ of $X$.

\smallskip
\textbf{(B)} 
Is $L$ invertible, up to some ``natural'' linear space which we expect to be its kernel? Eventually, we want to be stably invertible, which is not true because timelike singularities are invisible for $L$ \cite{MR1004174, LOSU-strings, LOSU-Light_Ray}. For some classes of $f$, that might still be true. We expect that  $Lf=0$ with $f\in C_0^\infty$ at least, implies that $f$ is a sum of a potential field and one conformal to $g$. This is the expectation based on the linearization of the (known so far) non-uniqueness of ${\mathcal S}^\sharp$. 
This kind of injectivity is proved  in \cite{Lauri-Light-21} when  $g$ is conformal to $\d t^2-h(x, \d x)$ with $h$ Riemannian, under the assumption that for $X_h$ related to $h$, we have injectivity modulo potential fields, even for tensor fields of orders $m>2$. This implies the same result for $g$ Minkowski of course.  

\smallskip
\textbf{(C)} We prove next scattering rigidity for one of those special cases where we expect it to hold: for stationary metrics, under some additional geometric conditions (after all, conditions are needed even in the Riemannian case). We use the observation made in \cite{Germinario_07} that once one projects the lightlike geodesics on the spatial base, one gets a magnetic system in space. Boundary rigidity for magnetic systems was studied in \cite{St-magnetic}, see also \cite{Yernat-Hanming-15,Hanming-18}. The equivalent to the boundary distance function there was taken to be the boundary action function $\mathbb A$, see also Appendix~\ref{app}. We show that $\mathbb A$  appears naturally when we reduce the Lorentzian scattering relation to knowledge of a defining function of pairs of boundary points which can be connected by a lightlike geodesic; and then project to the spatial base. Then all rigidity results in \cite{St-magnetic} apply and imply rigidity for stationary metrics. 

\smallskip
The boundary and the lens/scattering rigidity of Riemannian manifolds have rich history. It goes back to 1905 and 1907,  when Herglotz \cite{Herglotz} and Wiechert and Zoeppritz \cite{WZ} resolved the conformal spherically symmetric case motivated by seismology. The conformal case for simple metrics was solved in \cite{Mu2} and \cite{MuRo}. Further results can be found in \cite{Croke90, Otal, PestovU, bernstein1980conditions, Gromov, SU-MRL, CDS, LassasSU, Burago-Ivanov, SU-JAMS, SUV_localrigidity, SUV_anisotropic}. The lens rigidity problem, more appropriate for non-simple geometries is studied in \cite{Croke04, Michel, SU-lens, Croke_scatteringrigidity, Colin14}. 

Inversion of the geodesic X-ray transform on tensors on Riemannian manifolds has been well studied as well, see \cite{Denisjuk_06,Sh-book,S-AIP,S-Serdica,SU-AJM,SU-Kawai,SU-Duke,SUV-tensors,Monard-2016,UV:local} generalizing its version on functions. 

There are no so many results for Lorentzian geometry. The author and Yang \cite{St-Yang-DN} showed the  scattering relation appears as the canonical relation of the associated Dirichlet-to-Neumann map, which is an FIO. A linearization of the scattering rigidity problem from a spacelike to a spacelike hypersurface was studied in \cite{LOSU-strings} motivated by a problem in cosmology. It was shown that microlocally, a vanishing linearization implies that the perturbation is a sum of a conformal tensor plus a potential part, indeed. The recent paper \cite{UhlmannYangZhou} studies recovery of stationary metrics from the time separation function under an additional condition on the form $\omega$, see section~\ref{sec_stat}, but the data there uses information coming from not lightlike geodesics only. Of course, any result about rigidity of Riemannian manifold $(N,h)$ implies rigidity for the Lorentzian one $M=\R_t\times N_x$ with $g=-\d t^2+h(x,\d x)$. Microlocal study of the light ray transform on functions was done in \cite{LOSU-Light_Ray}. 

Light ray transform results in the Lorentzian setting exist as well: \cite{MR1004174, S-support2014,Feizmohammadi2019, wang2017parametrices, LOSU-Light_Ray,Vasy-Wang-21} for functions, \cite{Siamak2016} for one-forms, and \cite{Lauri-Light-21,LOSU-strings} for higher order tensor fields. 
One major distinction is that one can recover spacelike singularities,  probably say something about the lightlike ones, but the timelike ones are hidden. This can be interpreted as ability to see signals moving slower than light (or sound, etc.) but not ones moving faster than it. On the other hand, in some specific situations motivated by physics, signals moving faster than light should not exist (well, ignoring the discussion what a signal is, and ignoring the distinction between phase and group velocities at the moment). This has been used in \cite{Vasy-Wang-21} to show that  if $f$ solves a wave equation with speed one or less, one can recover all singularities, and in fact invert the light ray transform stably on a subclass of functions. The light ray transform on two-tensor fields however remains not well understood, and even its relation to the scattering relation was not clear. Clarifying the latter is one of our goals here. 

A few words about the conventions in this paper. All functions or maps are  smooth. By $\langle \omega,v\rangle$ we denote the action of the covector $\omega$ on the vector $v$, sometimes denoted by $\omega(v)$ in the literature. For a covariant tensor $f$ or order two, $\langle f,v\otimes w\rangle = f_{ij}v^iw^j$ in local coordinates, sometimes denoted as $f(v,w)$ in the literature.  The notation $(u,v)_g$ is reserved for the scalar products of two vectors in the metric $g$. For a vector $v\in T_xM$ with a base point$x$  on the boundary $\partial M$, $v'$ stands for its orthogonal projection to $T_x(\partial M)$. Note that this makes sense as long as $T(\partial M)$ is either time or space like. In section~\ref{sec_stat}, in local coordinates $(t,x)$, we will write $v=(v_t,v_x)$ for a vector $v$, where $v_t$ is the zeroth (time) component of $v$, while $v_x$ is the $n$-vector consisting of the spatial components, not to be confused with partial derivatives.  

\smallskip 
\textbf{Acknowledgments.} 
This work was partly inspired by and started during the author's stay at the RICAM (Linz, Austria) in the fall of 2022. 

\section{The defining function $r$ of $\Sigma$ and its linearization} \label{sec2}
\subsection{The Riemannian case} \label{sec_R}
Let $(M,g)$ be a compact Riemannian manifold of dimension $n$ with boundary. The lens rigidity problem is to recover the metric (and the topology, if unknown) from the lens data 
 $(\mathcal{S},\ell)$ consisting of the images $(y,w)=\mathcal{S}(x,v)$ of all boundary points $x$ and unit incoming directions $v$, where $y$ is the outgoing point (assuming $\gamma_{x,v}$ non-trapping), and $w$ is the outgoing direction of the geodesic $\gamma_{x,v}$ with $\ell(x,v)$ being its length. The map $\mathcal S$ alone is called the scattering relation.\footnote{in author's view, the lens and the scattering data notions should be swapped: lens data should refer to $\mathcal S$ only, with no reference of time of propagation, while (time-dependent) near field scattering data should include the latter. In time-space, time is already included, so calling $\mathcal S$ scattering data is justified.} 
This recovery is expected to be done up to an isometry fixing the boundary $\partial M$ pointwise. If it holds, then $(M,g)$ is called lens rigid. It is a well studied problem, as we pointed out in the introduction. The boundary rigidity problem has the same goal but the data is the geodesic distance $\rho(x,y)$ between boundary points. Under ``simplicity'' assumptions, we have $v=\exp_x^{-1}(y)/  |\exp_x^{-1}(y)|$, which allows us to express $\rho$ through $\ell$ as 
\be{0}
\rho(x,y) = \ell\left(x, \exp_x^{-1}(y)/  |\exp_x^{-1}(y)|\right). 
\ee

A better way to think about the parameterizations of the geodesics leaving and arriving at $\partial M$ is to project $v$ and $w$ on $T_x(\bo)$ and $T_y(\bo)$, respectively; let $v'$ and $w'$, respectively, be those projections. They determine $v$ and $w$ uniquely. Then we can view $\mathcal S$ and $\ell$ as maps from $(x,v')\in B(\bo)$ (the tangent unit ball bundle) to $(y,w')$ which belongs to the same bundle. Only then will  $\mathcal S$ and $\ell$ be invariant under isometries as above. 
 
The so redefined map  $\mathcal S$ is then symplectic when lifted to $\mathcal S^\sharp$ on  $B^*(\bo)$. When $x_0\in\bo$ and $y_0\in\bo$ are not conjugate to each other along some geodesic, $\rho$ is well defined near $(x,y)$ when the distance is the geodesic length restricted to a neighborhood of that geodesic, and 
\be{1}
\mathcal{S}^\sharp \left(x,-\d_x'\rho(x,y)\right) = \left(y,\d_y'\rho(x,y)\right),\quad (x,y)\in \bo\times\bo,
\ee
where $\d'$ stands for the tangential projection of $\d$ onto $T^*(\bo)$. The same formula can be written for $\mathcal{S}$ by replacing $\d$ by $\grad$, which requires knowledge of $g$ on the boundary. 
Formula \r{1} is an observation by Michel \cite{Michel}, see also \cite{S-Serdica}.  Therefore,  $\mathcal S^\sharp$  which maps $\R^{2n-2}$ to itself locally  is actually determined by the derivatives of a (scalar function) mapping $\R^{2n-2}$ to $\R$ locally.   
 
It is straightforward to see that under the non-conjugacy assumption, $\rho$ determines  $(\mathcal S^\sharp,\ell)$  locally, and vice versa. Indeed, knowing $\rho$, we can recover $\mathcal S^\sharp$ by \r{1}, and then $\ell$ by \r{0}. On the other hand, given  $\mathcal S^\sharp$, for $(x,y)\in \partial M\times \partial M$ fixed, we can recover the projection $\xi'$  of the incoming codirection  by \r{1}. Then we recover $\rho(x,y)$ by \r{0} knowing $\ell$. We want to mention that the non-conjugacy assumption makes  $\mathcal S^\sharp$ a ``free'' canonical transformation in the terminology of \cite[chapter~47]{Arnold_Mech}, and then guarantees the existence of a generating function, which happens to be $\rho^2/2$ here. When $(M,g)$ is simple, in fact $\mathcal S^\sharp$ suffices to recover $\rho$. Indeed, by the arguments above, for a fixed $x\in\bo$, we know $\d'_y\rho(x,y)$ for all $y$. We can integrate that along a curve on $\bo$ connecting $x$ and $y$ to recover $\rho(x,y)$.

One of the approaches to boundary/lens rigidity is to linearize near a fixed metric and try to invert stably the resulting linear transform first.  A simple variational argument, see, e.g., \cite{S-Serdica}, shows that the linearization of the boundary distance function leads to the quite nicely looking geodesic X-ray transform of symmetric tensor fields of order two:
\be{2}
\delta \rho(x,y) = X(\delta g) (x,\exp^{-1}_x(y)),
\ee
with $Xf$ defined in \r{3}, 
where $f$ is a two-tensor field, and one often normalizes $v$ to unit vectors. 
We want to invert it stably up a potential field $\d^sv$ with $v=0$ on $\bo$, where $(d^s v)_{ij}= \frac12(v_{i,j}+v_{j, i})$ is the symmetrized differential. Such potential fields linearize the non-uniqueness of the nonlinear problem due to isometries. The geodesic X-ray transform of symmetric tensor fields has been studied extensively, as we pointed out in the Introduction. 
 
As we mentioned in the Introduction, if we linearize  $\mathcal S^\sharp$  instead, we get a not so simple looking formula. We still get a geodesic X-ray transform of symmetric tensor fields but of $\nabla f$, with a weight, plus a zeroth order term, see \cite{SU-MRL, SUV_localrigidity, SUV_anisotropic}.

\subsection{The Lorentzian case}  
Assume that we have a Lorentzian manifold $M$ of dimension $1+n$ now with a metric of signature $(-,+,\dots,+)$. One example is $g = -\d t^2+h_{\alpha\beta}(x)\d x^\alpha \d x^\beta$ with $h$ Riemannian, which leads to essentially the same problem as before, so it is not so interesting but is a good start to understand how the scattering rigidity problem would be formulated in that case. The setup above leads naturally to light rays starting from a cylindrical boundary $\R_t\times N_x$ and ending up there, where $(N,h)$ is compact with boundary, the equivalent of $M$ above. Assuming a general Lorentzian metric and lateral boundary which is timelike (or a spacelike), we get a scattering relation. One of the new features is that light rays do not have natural parameterization: their ``speed squared''  $(\dot\gamma, \dot\gamma)_g$ is zero, and rescaling the parameter along each one still yields a null geodesic. Moreover, given a family of such light rays, we may rescale by a factor changing from geodesic to geodesic, and obtain a different parameterization, no better or worse than the initial one. In any case, fixing some parameterization locally, a linearization of $\mathcal S^\sharp$  near a fixed light ray would produce a transform at least as unpleasant as in the Riemannian case. On the other hand, the much simpler form \r{LRT} of $L$ has already being studied with the anticipation that it must have something to do with the linearization of the Lorentzian scattering rigidity problem. 

The natural questions then are the following. 
Under a non-conjugacy assumption, can we define an equivalent $r(x,y)$ of the boundary distance function so that
\begin{itemize}
\item[(i)] \r{0} and \r{1} hold in some form, 
\item[(ii)] in particular, $r$ and $\mathcal S^\sharp$ determine each other,
\item[(iii)]  {the linearization of $r$ near some $g$ gives us $Lf$ defined in \r{LRT}}.
\end{itemize}

The problem with the attempt to choose $r$ to be a distance function  as a direct analog of $\rho$ is that given two points, they may not be connected by a (unique or not) lightlike geodesic. This is a property we really want in order to get \r{3}, and relate it to $\mathcal{S}$. If they are, then their geodesic distance is zero. The Lorentzian distance, see, e.g., \cite{Alias_2009}, does not seem to give a direct answer, either: it is zero on one side of the light cone (away from the chronological future), and singular at the light cone, where it vanishes. 
 
The  solution we propose is the following. Given a timelike lateral boundary $\bo$, generalizing the cylinder $\R\times \partial N$ above, the set of points connected by a unique lightlike geodesic is a submanifold of $\bo\times\bo$ of codimension one locally under a non-conjugacy assumption. Fix any defining function $r(x,y)$ of it. Then $r$ satisfies (i), (ii), (iii) above in an appropriate sense. Of course, $r$ is defined up to an elliptic factor only. On the other hand, null directions can be parameterized up to a scaling factor, so those two peculiarities correlate well.

\subsection{Main result about the defining function}
Our point of view here  is local, so for this reason, we assume that we work in an ambient  manifold with a complete  Lorentzian metric $g$, and we have two ``small'' either timelike (will be referred to as the (T) case), or spacelike (the (S) case) surfaces $U$ and $V$ corresponding to initial and endpoints, respectively. One can assume one of those cases on one side, and the other one on the other one, as well. 
We fix $x_0\in U$, $y_0\in V$ so that they are connected by a lightlike geodesic $[0,1]\ni t\to \gamma_0(t)$. Assume that 
\be{NC}
\text{$x_0$ and $y_0$ are not conjugate along $\gamma_0$}, \tag{NC}
\ee
and that $\gamma_0$ is transversal to $U$ and $V$ at their only points of intersection, $x$ and $y$. 

In the (T) case, we fix a time orientation on $U$ that we call future pointing, see also \cite{St-Yang-DN}. Assume that $\gamma_0$ is future pointing at $x_0$, and we choose a time orientation on $V$ so that  $\gamma_0$ is future pointing at $y$ as well. We also fix orientation on $U$ and $V$ in classical case, calling the sides containing $\gamma_0$ interior, and the other ones exterior. In the (S) case, the classical orientation relates to time orientation: the interior side of $U$ is the future one; the interior side of $V$ is the past. 
Set $v_{0}:= \dot\gamma(0)$, $w_0:= \dot\gamma(1)$. 
Let $v_0'$, $w_0'$ be their orthogonal projections on $TU$ and $TV$, respectively; see \cite[chapter~2]{ONeill-book}. They must be timelike/spacelike depending on which case we have (T) or (S). 
Let $\mathcal{U}$, $\mathcal V$ be small 
timelike/spacelike conic neighborhoods of $(x_0, v'_0)$ in $TU$, and of $(y_0,w_0;)$ in $TV$, respectively. 
We define the \textit{scattering relations} $\mathcal S$, $\mathcal S^\sharp$ below, see also Figure~\ref{Lorentz_distance_fig1}.

\begin{definition} \label{def1}
The scattering relation $\mathcal{S}: \mathcal{U} \to \mathcal{V}$ is defined by $\mathcal{S}(x,v') = (y,w')$ as follows. Let $v$ be the lightlike vector at $x$ with orthogonal projection $v'$ on $T_xU$, pointing to the interior; then  $y\in V$ is the point where the geodesic $\gamma_{x,v}$ issued from $(x,v)$ meets $V$, and $w'$ is the orthogonal projection on $T_yV$ of its direction there. 

Identifying vectors with covectors by the metric $g$ restricted to $TU$, we define $\mathcal S$ on the cotangent bundle as well, we call it $\mathcal{S}^\sharp$.
\end{definition}

\begin{figure}[h!] 
  \centering
  \includegraphics[scale=1,page=1]{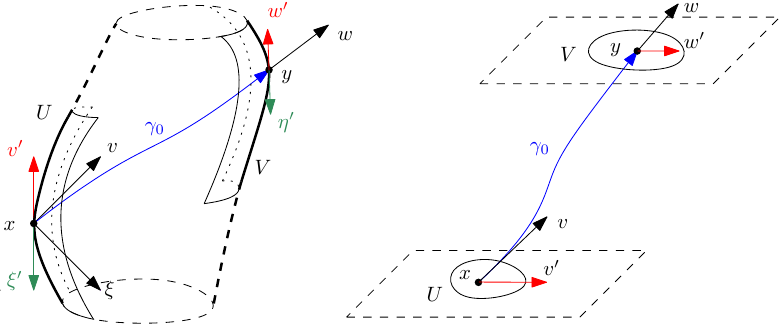}
\caption{\small  The scattering relation $\mathcal{S}(x,v')= (y,w')$ on the tangent bundle for $U$, $V$ either timelike (left) or spacelike (right), and its version $\mathcal{S}^\sharp(x,\xi')= (y,\eta')$ on the cotangent bundle (left only). 
}
\label{Lorentz_distance_fig1}
\end{figure}

 Clearly, $\mathcal S$ and $\mathcal{S}^\sharp$ are positively homogeneous of order one in the fiber variable. We have $\mathcal{S}(x,a(x,v')v') = (y,a(x,v')w') $ for every $a>0$. We may normalize $v'$ in some way to reduce the number of variables. For example, we may require $(v',v')_{g'}= \mp 1$ (recall that $v'$ is spacelike/timelike depending of whether $U$ is timelike/spacelike), where $g'$ is the induced metric. If $x^0$ is a local time variable in the former case, we may require $v^0=1$. More generally, we may 
restrict $v$ to some hypersurface so that each radial ray intersects it transversally. We call each such restriction of $\mathcal S$ a \textit{reduced representation} of $\mathcal{S}$, similarly for  $\mathcal{S}^\sharp$. Each of them is just a representation of $\mathcal S$ or $\mathcal{S}^\sharp$ modulo rescaling by positive factors, depending on $(x,v')$. 
Knowing the reduced version recovers the full one in a trivial way. 

\begin{remark}\label{rem_0}
The non-conjugacy and the transversality assumptions imply that  $\exp_x : \exp_x^{-1}(V)\to V$ is a diffeomorphism for every $x\in U$ as long as $U$ and $V$ are small enough. Assuming that, we can also assume that $\mathcal{U} =\{\exp_x^{-1}(V); \; x\in U\}$; then we have a global diffeomorphism for $(x,v)\in \mathcal{U}$ above. Finally, we can project $v$ to $T_xU$ to get a diffeomorphic map $T_xU\ni v'\to y\in V$ for every $x\in U$. We are going to consider geodesics (lightlike or not) issued from $(x,v)\in \mathcal{V}$ only. 
\end{remark}

\begin{definition}\ 

(a) The set $\Sigma\subset U\times V$ consists of pairs $(x,y)$ so that $x$ and $y$ are connected by a unique lightlike geodesic (locally). 

(b) The smooth function $r: U\times V\to \R$ is called a defining function of $\Sigma$, if (i) $r=0$, $\d r\not=0$ on $\Sigma$, and (ii) $r(x,y)<0$ if and only if the locally unique geodesic connecting $x$ and $y$ is timelike. 
\end{definition}

Condition (ii) is just a sign convention for $r$; negativity means that $y$ is in the chronological future of $x$.

\begin{theorem}\label{thm1}
With the assumptions above, and in particular, with the non-conjugate assumption \r{NC}, we have

(a) $\Sigma$ is a codimension one submanifold of $U\times V$ when $U$ and $V$ are small enough.  

\noindent Let $r:U\times V\to \R$ be any defining function of $\Sigma$. Then 

(b) $\left\{(x,-\d'_x r,y,\d_y'r);\; r(x,y)=0   \right\}$ coincides locally with the graph of some reduced representation of $\mathcal{S}^\sharp$.  

(c) If $g_\tau$ is an one-parameter family of Lorentzian metrics smoothly depending on $\tau$ near $\tau=0$, so that $g_0=g$, and $r_\tau$ are associated defining functions smoothly depending on $\tau$, we have
\be{thm1_1}
\frac{\d}{\d\tau} \Big|_{\tau=0}r_\tau(x,y) = \kappa(x,y) \int_0^1 \left\langle f, \dot \gamma_{[x,y]} (t)\otimes  \dot \gamma_{[x,y]} (t)\right\rangle\,\d t\quad \text{on $\Sigma= \{r(x,y)=0\}$},
\ee
where $f=\d g_\tau/\d \tau|_{\tau=0}$, $[0,1]\ni t\to\gamma_{[x,y]} (t)$ is the locally unique lightlike geodesic in the metric $g$ connecting $x$ and $y$, and $\kappa$ is a smooth non-vanishing function. 
\end{theorem}

Before proceeding with the proof, we make a few observations.
\begin{remark}
 The non-uniqueness of the defining function $r$ due to the freedom to multiply by any elliptic factor $\kappa(x,y)$ implies the following. If $r_\kappa=\kappa r$ is another defining function, then in (a), $\mathcal{S}^\sharp(x,\xi') = (y,\eta')$ with $\xi'=-\d_x'r$, $\eta'=\d'_yr$ 
is replaced by $\mathcal{S}^\sharp(x,\kappa \xi') = (y,\kappa \eta')$, which is just another reduced representation of $\mathcal{S}^\sharp$. Each one of them determines $\mathcal{S}^\sharp$ by homogeneity. Replacing $r_\tau$ by $\kappa_\tau r_\tau$ in \r{thm1_1} multiplies the right-hand side by $\kappa_0$. 
\end{remark} 

\begin{remark}
The non-trivial elliptic factor $\kappa$ in \r{thm1_1} is inevitable since the light ray transform on the right is determined up to rescaling of the parameter $t$, and the defining function is determined by such a factor as well, as we already emphasized. 
\end{remark}

\begin{remark}
Accounting for the homogeneity of $\mathcal S^\sharp$ with respect to the fiber variable, Theorem~\ref{thm1}(b) implies that the graph of $\mathcal{S}^\sharp$ (locally) coincides with the twisted conormal  bundle $N^*\Sigma'\setminus 0$. In particular, $\mathcal S^\sharp$ is a symplectic local diffeomorphism. The latter already follows from \cite{St-Yang-DN} since  the canonical relation of the associated DN map related to the wave equation $\Box_gu=0$, at least in the (T) case, is the graph of $\mathcal{S}^\sharp$. Under the non-conjugacy conduction \r{NC}, that FIO has a Schwartz kernel conormal to $\Sigma$, which can be shown using the parametrix. 
\end{remark}

\begin{corollary}\label{cor_2.1}
The scattering relation $\mathcal S^\sharp$ determines $r$ uniquely up to an elliptic factor. On the other hand, each defining function $r$ determines $\mathcal S^\sharp$ uniquely. 
\end{corollary}

\begin{example} \label{ex1}
The Riemannian case: Let $g=-\d t^2+h(x,\d x)$ on $M=\R\times N$, where $(N,h)$ is a compact Riemannian manifold with convex boundary. Multiplying $g$ by a positive conformal factor $\lambda(t,x)>0$ leaves $\mathcal{S}^\sharp$ unchanged, see section~\ref{sec_conf}, so this example covers the static case, in fact.

 Take $r(t,x,s,y) = t-s+\rho(x,y)$. Then $r=0$ if and only if $(s,x)$ and $(t,y)$ are connected by the lightlike geodesic $\Gamma:[ s, t]\ni \sigma \mapsto (\sigma,\gamma(\sigma))$, which we will parametrize by a unit length parameter, where $\gamma$ is the unique (locally, under  the non-conjugacy assumption) unit speed geodesic in the $x$-space, connecting $x$ and $y$. We have $\d'_{t,x}r =(1, \d'_{x}\rho)$, $\d'_{s,y}r = (-1, \d'_{y}\rho)$, where $\d'$ stands for the tangential gradient. Converting them to vectors using the associated Lorentzian metric, we get $-\grad'_{t,x}\! r =(1,- \grad'_{x}\! \rho)$, $\grad'_{s,y}\! r = (1, \grad'_{y}\! \rho)$. Those are exactly the fiber components of the Lorentzian so reduced (normalized) scattering relation, see  \r{1}, and also Theorem~\ref{thm_mag} below.  

One can also take $r_1= -(t-s)^2+\rho^2(x,y)$. In a neighborhood of a fixed pair on $\Sigma$, $r_1$ equals $r$ times a non-vanishing factor, and is therefore another defining function. The points of any  such pair $(t,x)$,  $(s,y)$ are connected by a unique (locally) geodesic, not necessarily lightlike. Indeed, one can take $\Gamma$ as above but now $\dot\gamma$ is not unit   anymore; we require $|\dot\gamma|_h= \rho(x, y)/(s-t)$. Then we can take a defining function $r_2$ to be the energy associated with $\Gamma$, which is $r_2=r_1/(s- t)^2$. 
\end{example} 

\begin{proof}[Proof of Theorem~\ref{thm1}]
To prove (a),  notice that for $x\in U$ fixed, $\exp_x(\cdot)$ maps the light cone to a smooth hypersurface in $M$ near $x_0$ by the non-conjugacy condition. By the transversality assumption, the intersection with $V$ is a codimension one submanifold of $V$. We can include $x$ in this argument by adding $n-1$ more dimensions, and complete the proof of (a).

We proceed with the proof of (b). We construct first a defining function using the notion of energy of a curve. 
For a smooth curve $[a,b]\ni t\to c(t)$, one defines the (non positive or negative definite) energy functional
\[
E(c) = \frac12 \int_a^b (\dot c(t), \dot c(t))_g\,\d t.
\]
For a smooth variation $H: [-\eps,\eps]\times [a,b]\to M$ of $c$ (so that $H(\tau,t)|_{\tau=0}=c(t)$), we have the following variational formula
\be{4}
\frac{\d}{\d \tau} \Big|_{\tau=0}E= (W,\dot c)_g\Big|_a^b - \int_a^b (W,D_t \dot c)_g\,\d t,
\ee
where $W=\partial_sH(0,\cdot)$ is the variation field of $H$, and $D_t$ is the $t$-covariant derivative. The proof is the same as in the Riemannian case since it depends on the calculus of the covariant derivatives and it is independent of the signature of the metric. In particular, each geodesic, lightlike or not,  is a critical point of $E$ under proper variations (those fixing the endpoints). 

Let $g_\tau$ be the family of Lorentzian metrics as in the theorem. Let $\gamma_\tau(t)$ be the unique geodesic in the metric $g_\tau$ connecting $x$ and $y$. We parameterize it by $t\in [0,1]$. The associated energy in the metric $g_\tau$ is given by 
\be{5}
E_{g_\tau}(\gamma_\tau) = \frac12 \int_0^1 (\dot \gamma_\tau(t), \dot \gamma_\tau(t))_{g_\tau}\,\d t.
\ee
If we fix $\tau$ in $g_\tau$ on the right (only), the $\tau$-derivative  would vanish since each geodesic with fixed endpoints is a critical point of the energy functional. Thus we get
\be{6}
\frac{\d}{\d \tau} \Big|_{\tau=0}E_{g_\tau}(\gamma_\tau) =  \frac12 \int_0^1 \big\langle f , \dot\gamma_0\otimes\dot\gamma_0  \big\rangle \,\d t,\quad f:=\frac{\d}{\d \tau} \Big|_{\tau=0} g_\tau. 
\ee
With this in mind, we define the following function
\be{7}
r(x,y) = E(\gamma_{[x,y]}),\quad (x,y)\in U\times V,
\ee
where $[0,1]\ni t\to \gamma_{[x,y]}(t)$ is the unique (locally) geodesic connecting $x$ and $y$. We claim that $r$ is a defining function of $\Sigma$. Indeed, the integrand in \r{6} with $\gamma_0 = \gamma_{[x,y]}$ there, is constant along $\gamma_{[x,y]}$; it is zero if and only if $\gamma_{[x,y]}$ is lightlike and then $(x,y)\in \Sigma$ by definition. The integrand is negative, if and only if $\gamma_0$ is timelike, and then $r>0$, as required. We will show that $\d'_{x,y} r\not=0$ on $\Sigma$ as a byproduct of the analysis below.

Let $W$ in \r{4} correspond to variations of $\gamma_{[x,y]}$, see also \r{7}, with $x\in U$ fixed and $y$ varying near $y_0$ in $V$. Then $(\d/\d \tau) r(x,y(\tau)) = (\partial y/\partial\tau,\dot\gamma(1))_g$ at $\tau=0$ by \r{4}. On the other hand, we have $(\d/\d \tau) r(x,y(\tau)) =   \langle\d_y r, \partial y/\partial\tau\rangle $ at $\tau=0$. Converting $ \d_y r$ to a vector by the metric, we see that $\grad_y r(x,y)$ and $\dot\gamma(1)$ have the same projection on $T_yV$. Now, we can fix $y$ and vary $x$ to get a similar conclusion but with a negative sign coming from \r{4}.

Next, we get $ \d_y' r\not=0$, also $\d_x' r\not=0$, because $\dot \gamma'(1)\not=0$ since that projection is either timelike in the (T) case, or spacelike in the (S) case.    This completes the proof of (b). 

Finally, \r{thm1_1} with $\kappa=1/2$ follows from \r{5} when $r=r_\tau$ is as \r{7} for each $\tau$. A different defining function of $\Sigma_\tau$ would be of the form $r_1 =  \kappa_\tau r_\tau$ with some elliptic $\kappa_\tau$, and its $\tau$ derivative at $\tau=0$ just gains the factor $\kappa_0$ (multiplied by $1/2$). This proves (c). 
\end{proof}

\section{Gauge invariance of the scattering relation and of the light ray transform} \label{sec_gauge}
In this section, we collect some known fact about the two maps, see, e.g., \cite{LOSU-strings}. 

There are several obvious groups of transformations which leave the scattering relation invariant, and their linearizations are in the kernel of $L$. 

\subsection{Invariance under diffeomorphisms} \label{sec_conf}
Let $\psi$ be a local diffeomorphism from a neighborhood of $\gamma_0$ to its image in $M$. Assume $\psi|_U=\Id$, $\psi|_V= \Id$ in the setup in section~\ref{sec2}. Then $\mathcal{S}_g = \mathcal{S}_{\psi^*g}$ and  $\mathcal{S}_g^\sharp = \mathcal{S}_{\psi^*g}^\sharp$ in a trivial way. 

The linear counterpart of this is the following. Assume we have a smooth one-parameter family of 
diffeomorphisms $\psi_\tau$ fixing $U$ and $V$ near $\tau=0$ with $\psi_0=\Id$. Then for $g_\tau := \psi_\tau^* g$ we have $(\d/\d t) g_\tau|_{\tau=0} = 2\d^s v$, where $v=(\d/\d t) \psi_\tau|_{\tau=0}$. Note that $v=0$ on $U$ and on $V$. Moreover, all $v$'s like that are possible linearizations of one-parameter groups of diffeomorphisms fixing $U$ and $V$. 
Therefore, we get that $L(\d^s v)=0$ for all such $v$'s. This can be verified independently by applying the Fundamental Theorem of Calculus to the identity
\[
\frac{\d}{\d t} (v , \dot\gamma(t))_g = \langle \d^s v, \dot\gamma\otimes\dot\gamma \rangle,
\]
where on the right, $v$ is identified with its covector version by lowering the indices. This is a well known fact in tensor tomography on Riemannian manifolds, at least. 

\subsection{Invariance under conformal changes} Let $\tilde g = c(x)g$, where $c(x)>0$ is a smooth function. Then lightlike/spacelike/timelike vectors or covectors in the metric $g$ are such in the metric $\tilde g$ as well. Moreover, the lightlike geodesics in the metric $g$ remain lightlike in the metric $\tilde g$ as well as curves but with a changed parameterization preserving the direction (in general, they do not solve the geodesic equation). 

The easiest way to prove this is to pass to the Hamiltonian formalism. An alternative proof is given in  \cite[Lemma~6.1]{LOSU-strings}. 
With $H=\frac12 g^{ij}(x) \xi_i \xi_j$, the Hamiltonian curves  $(x(s),\xi(s))$ at the level $H=0$, transformed into curves on the tangent bundle coincide with the lightlike geodesics. For $\tilde H$, associated with $\tilde g$, we have the system
\[
\dot {\tilde x}^i = c^{-1}g^{ij}\tilde \xi_j, \quad  \dot {\tilde \xi}_i = -\frac12  c^{-1} \partial_{x^i}g^{kl}\tilde \xi_k\tilde \xi_l + \frac12 (\partial_{x^i}c) g^{ij}(x) \tilde\xi_i \tilde\xi_j.
\]
In addition, we have initial conditions $(x(0),\xi(0))=z\in\mathcal U$; we will denote then the solutions by $(x(s,z),\xi(s,z))$. 
On $\tilde H=0$, the last summand vanishes; therefore we are left with 
\be{3.1}
\dot {\tilde x}^i = c^{-1}g^{ij}\tilde \xi_j, \quad  \dot {\tilde \xi}_i = -\frac12  c^{-1} \partial_{x^i}g^{kl}\tilde \xi_k\tilde \xi_l.
\ee
Assume initial conditions at $s=0$. 
Let $\alpha$ solve
\[
\dot\alpha(s,z)= c^{-1}(x(s,z)), \quad \alpha(0,z)=0,
\]
where $z\in \mathcal{U}$ is the initial condition, $x(s,z)$ is the $x$-component of the solution of \r{3.1} with $c=1$ (those quantities have no tildes over them). Then $(\d/\d s)x(\alpha(s,z))= c^{-1}\dot x$, and similarly for $\xi(\alpha(s,z))$, with initial condition $z$ at $s=0$. Comparing this to \r{3.1}, we conclude $(\tilde x, \tilde\xi) = (x,\xi)\circ\alpha(s,\cdot)$. Note that this conclusion presumes the same initial conditions at $s=\tilde s=0$. 

We show next that $\mathcal{S}^\sharp = \tilde{\mathcal{S}}^\sharp $. 
This follows from the fact that for $(x,\xi)$ so that $(x,\xi^\sharp)$ is as in Definition~\ref{def1}, the exit points and codirections related to $g$ and $\hat g$ are independent of the reparameterization but they happens for possibly different values of $s$: $s_0$ and $\tilde s_0$ so that $\alpha(\tilde s_0)=s_0$. 

%
%
 
On the other hand, one can write $\mathcal{S} = g^{-1} \circ \mathcal{S}^\sharp  \circ  g$; therefore, $\tilde{\mathcal{S}} = c^{-1}g^{-1} \circ \tilde{\mathcal{S}}^\sharp  \circ gc = c^{-1}g^{-1} \circ {\mathcal{S}}^\sharp  \circ gc = c^{-1} \circ \tilde{\mathcal{S}} \circ c$, where $g$ denotes the operator $(x,v)\mapsto (x,gv)$, similarly for the other multiplication operators there.  Therefore, $\mathcal{S}$ is \textit{not} invariant under general conformal changes. 

To linearize this, assume $g_\tau = c_\tau g$. Then the linearization of $g_\tau$ at $\tau=0$ is $c(x)g$ with  $c=(\d c/\d \tau) c_\tau|_{\tau=0}$. Thus $L(cg)=0$. This is obvious by itself since the integrand in \r{LRT} vanishes pointwise when $f=cg$. 

\subsection{Lens rigidity and light ray transform injectivity formulations} 
With the above in mind, we can formulate the scattering rigidity problem as follows. Show that $\mathcal{S}^\sharp_{g_1}=\mathcal{S}^\sharp_{g_2}$ implies $g_2= c \psi^*g_1$ with $c>0$, and a diffeomorphism $\psi$ fixing $U$ and $V$. We are vague on purpose here about the assumptions and the region we expect to prove that equality since we have various cases even in the Riemannian case. We want to emphasize that we use $\mathcal{S}^\sharp$ as our data, not $\mathcal{S}$. 

The injectivity of $L$ under the gauge can be formulated like this: show that under some assumptions, $Lf=0$ implies $f=\d^sv+\lambda g$, where $v=0$ on $U$ and on $V$, and $\lambda$ is a scalar function. 

Both problems are open, and in section~\ref{sec_stat}, we will consider the special case of stationary metrics for the scattering rigidity problem.

\section{Rigidity of stationary metrics}\label{sec_stat}
We consider stationary metrics in this section. We refer to \cite{Lauri-Light-21, UhlmannYangZhou} for a justification of the interest in such metrics. 
\subsection{Stationary spacetime geometry} 
\subsubsection{Stationary metrics}
In $\R^{1+n}$, consider metrics of the form
\be{5.1}
g = -\lambda (x)\d t^2+2\tilde \omega_j(x) \d t \,\d x^j 
+\tilde h_{ij}(x)\d x^i \d x^j
\ee
with $\lambda>0$, $\tilde \omega= \tilde \omega_j\d x^j$ an 1-form in $\R^n$, and $\tilde h$ a symmetric tensor on $\R^n$; all time-independent. 
In matrix form,
\[
g= 
  \begin{pmatrix}
    -\lambda & \tilde \omega_1 & \dots & \tilde \omega_n \\
    \tilde \omega_1 & \tilde h_{11} & \dots & \tilde h_{1n}  \\
    \vdots & \vdots & \ddots & \vdots \\
    \tilde \omega_n & \tilde h_{n1}  & \dots & \tilde h_{nn} 
  \end{pmatrix}.
\]
Since we want $g$ to be Lorentzian, it is convenient to complete the square, and after useful rescaling of $\omega$ and $h$ by $\lambda$, write $g$ in the form
\be{5.2}
g= \lambda (x)\left( -\left(\d t+ \omega_j(x)  \d x^j\right)^2 + h_{ij}(x)\d x^i \d x^j\right) 
\ee
with $ h_{ij}=\lambda^{-1}\tilde  h_{ij} + \lambda^{-2}\tilde \omega_i \tilde \omega_j$, assumed positive definitive, and $\omega = -\lambda^{-1}\tilde\omega$. Occasionally, we will use the notation $g_{\lambda,\omega,h}$ for a metric of the kind \r{5.2}. The metric $h$ would be positive definitive,  if  $ \tilde h$  is positive definite as well. 
Given a Lorentzian manifold, one can derive that form of the metric globally as well, where $\omega$ and $\lambda$ are invariantly defined, from abstract assumptions of global hyperbolicity, and the existence of a complete timelike Killing field, see, e.g,  \cite{Lauri-Light-21, StrohmaierZ}. We are not going to go into details of that and just will assume that our Lorentzian manifold is $\R_t \times N_x$, with $(N, h)$ Riemannian, and that the Lorentzian metric there is given by \r{5.2} locally, which is actually a global definition assuming $\omega$ a well-defined one-form on $N$. 

The dot product $(\cdot,\cdot)_g$ can be derived from \r{5.1} by polarization. For $v=(v_t,v_x)$, $w=(w_t,w_x)$, we get
\be{dot}
(v,w)_g = \lambda \big( -(v_t+\langle \omega, v_x\rangle ) (w_t+\langle \omega, w_x\rangle ) +\langle h, v_x\otimes w_x\rangle \big).
\ee

\subsubsection{Invariance of the scattering relation and gauge equivalence}
There are two obvious groups of diffeomorphisms keeping $g$ in the form \r{5.2}, and keeping the scattering relation $\mathcal{S}^\sharp$ the same. First, for any diffeomorphism $\psi:N\to N$ fixing $\partial N$ pointwise,  setting $\Psi:= \Id\otimes\psi$, we have that $\Psi^*g$ is of the same form with $h$, $\omega$ and $\lambda$ replaced by $\psi^* h$, $\psi^*\omega$ and $\psi^* \lambda$. Next, it is easy to see that adding an exact form to $\omega$ provides an isometric metric. Indeed, let $\phi(x)$ be a smooth function. Then with $\tilde \omega:= \omega+ \d\phi$, we have 
\[
\d t+ \omega_j(x)  \d x^j = \d(t+ \phi(x)) +   \tilde  \omega_j(x)\d x^j,
\]
therefore, keeping $x$ the same and doing the change 
\be{t-change}
t'=t+\phi(x)
\ee
yields a metric like \r{5.2} but with $\tilde w$ instead. As we see below, at least locally, only $\d \omega$ matters for the projections of the lightlike geodesics on $M$, consistent with the observation we just made. Therefore, for the diffeomorphism $\Phi(t,x)= (t+\phi(x),x)$, we get $\Phi^*g$ of the same form with $h$ and $\lambda$ the same, and $\omega$ replaced by $\omega+ \d\phi$. Since we want this diffeomorphism to fix $\R\times N$ pointwise, we will require $\phi=0$ on $\partial N$. 
This observation also shows that while there is a well defined time direction $\partial/\partial t$ preserving the geometry, roughly speaking (a Killing vector field), there is no natural time variable $t$, say defined up to translations and time reversal since we  have time shifts depending on $x$ preserving the form of the metric as well. 

Finally, a conformal factor $\mu>0$ keeps the scattering relation $\mathcal{S}^\sharp$  intact, as well. This suggests the following.

\begin{definition}\label{def_gauge}
The metrics $g_{\lambda,\omega,h}$ and $g_{\hat \lambda,\hat\omega,\hat h}$ are called gauge equivalent, if there exists a diffeomorphism $\psi:N\to N$ fixing $\partial N$ pointwise, and a function $\phi$ vanishing on $\partial N$, 
so that
\be{group}
 \hat \omega = \psi^*(\omega +\d\phi), \quad  \hat h = \psi^*h. 
\ee
\end{definition}
The transformations \r{group} determined by $(\psi,\phi)$ form a group with generators the two elementary transformations above. 
Application of $(\psi_1,\phi_1 )$, and then $(\psi_2,\phi_2 )$ is of the same kind with $\psi=\psi_2^*\circ \psi^*_1$, $\phi = \phi_2+\psi_1^*\phi_1$.

Since we are interested in scattering rigidity (related to lightlike geodesics), by the results of section~\ref{sec_gauge}, we can replace $g$ by $\lambda^{-1}g$ having the same scattering relation.
Thus without loss of generality we can assume $\lambda=1$. Then 
\be{5.3x}
 g=- \left(\d t+\omega_j(x)  \d x^j\right)^2 +   h_{ij}(x)\d x^i \d x^j,
\ee
and lightlike vectors $v=(v_t,v_x)$ are characterized by 
\be{5.3xx}
-(v_t+\langle \omega,v_x\rangle)^2 +|v_x|_h^2=0.
\ee
It is convenient to normalize the parameterization along the lightlike rays $(t(\sigma),x(\sigma))$ by requiring $|\dot x|_h=1$. 

\subsubsection{Orthogonal projection on the boundary} \label{sec_proj} 
There is a natural projection $\pi:M = \R\times N \to N$, invariant under changes \r{t-change}, and $N$ can be considered as the manifolds of orbits generated by the Killing vector field $\partial/\partial t$, see, e.g., \cite{StrohmaierZ}.  It generates a projection $\d\pi$ between the tangent bundles. 
It is useful to understand orthogonal projections to $\R\times\partial N$ next. Let $(x,v)\in TN$ be such that $x\in\partial N$, and $v$ is pointing to the exterior of $\R\times N$. Eventually, we will apply this to $(y,w)$ and to $(x,-v)$ in the notation of the scattering relation. Let $\nu_x$ be the exterior unit normal to $\partial N$ in the metric $h$. 
It is straightforward to show, using \r{dot}, that $\nu:= (-\langle \omega, \nu_x\rangle, \nu_x)$ is normal to $\R\times \partial N$, spacelike in particular, exterior, and unit in the sense $(\nu,\nu)_g=1$. 
The projection $v'$ under question is given by $v' =v-( v,\nu)_g\, \nu $ with $\nu$ as above. 
We have $( v,\nu)_g = (v_x,\nu_x)_h$, therefore, in local coordinates,
\be{4.8v}
v' = \left( v_t+ \langle \omega, \nu_x\rangle(v_x,\nu_x)_h  , v_x'\right), 
\ee
where $v_x'$ is just the orthogonal projection of $v_x$ on $T(\partial N)$. 

We will further decompose $v'$ in the following way. We write
\be{5.3xxx}
v'=[v'_t,v_x'],
\ee
where $v_t' = -(v', \partial/\partial t)_{g'}$ is the (scalar) orthogonal projection of $v'$ to $\partial/\partial t$ in the induced metric $g'$, and $v_x'=(\d\pi) v'$. It is easy to see that $v_x'$ is as before, which explains the same notation but the emphasis now is that it has an invariant meaning. For $v_t'$, we get 
\be{5.3y}
v_t'= v_t+ \langle \omega, \nu_x\rangle(v_x,\nu_x)_h + \langle \omega',v'_x\rangle = v_t+\langle \omega,v_x\rangle,
\ee
 where $\omega'$ is $\omega$ restricted to $\bo$. We want to emphasize that \r{4.8v} is a coordinate representation, and while some terms are invariantly defined, $v_t$ and $\langle \omega, \nu_x\rangle$ are not; the latter, for example, is not preserved under coordinate changes \r{t-change}. On the other hand, \r{5.3xxx} has an invariant meaning. 

Finally, if $v$ is lightlike and future pointing, then $v_t'=|v_x|$ by \r{5.3xx}.

It is useful to introduce boundary normal coordinates as in the proposition below. 
\begin{proposition}\label{pr_5.1}
Let $(t_0,x_0)\in \R\times \partial M$. Then there exist local coordinates in $\R\times M$ near $(t_0,x_0)$ in which $g$ in \r{5.3x} takes the form
\[
g = -(\d t+ \omega_\alpha(x) \,\d x^\alpha)^2+ h_{\alpha\beta}(x)\d x^\alpha \d x^\beta + (\d x^n)^2
\]
with some   $ \omega$; and $\R\times \partial N$ is given by $\R\times \{x^n=0\}$. Summation over Greek indices is taken from $1$ to $n-1$. 
\end{proposition}

\begin{proof}
We put $h$ in boundary normal coordinates first; the construction is well known. Then we seek $\phi$ so that $(\omega-\d\phi)_n=0$. The latter is equivalent to $\partial_{x^n}\phi=\omega_n$, which we solve with the initial condition $\phi=0$ for $x^n=0$. Then we set $t'= t-\phi(x)$. In the coordinates $(t',x)$, we get the desired form.
\end{proof}

In those coordinates, $\nu=(0,\nu_x)$, $v_t'=v_t+\langle \omega',v'\rangle$.

\subsection{Reduction to a magnetic system} 
We review some results in \cite{Germinario_07}, see also \cite{Flores2002}. 
The geodesics of $g$ as in \r{5.2} there, projected to $N$ under $\pi$, see section~\ref{sec_proj} and Figure~\ref{Lorentz_distance_fig2} below, are characterized as the integral curves of a certain magnetic system. Note that if we include a non-constant conformal factor, as in \r{5.1} or \r{5.2}, and we consider not null geodesics only, then the system has an additional electric potential, see also \cite{Yernat-Hanming-15}. 

For the computations below, recall that if $\alpha$ is an one-form, then $d\alpha$ is a two-form satisfying $\langle \d\alpha, X\otimes Y\rangle = \langle \nabla_X\alpha, Y\rangle - \langle\nabla_Y\alpha, X\rangle $, where $\nabla$ is the covariant differential, and this is true independently of the background metric. 

It follows from \r{4} that a smooth curve connecting two points is geodesic if and only if it is a critical point of the energy functional without being a minimum or a maximum. Let $W = (T,X)$ be a variation of a  geodesic $[0,1]\ni \gamma(\sigma)= (t(\sigma),x(\sigma))$ fixing the endpoints. The energy form takes the form  
\be{5.3a}
E = \int_0^1\left(-\big( (\dot t +\langle\omega,\dot x\rangle\big)^2 +|\dot x|_h^2 \right) \d \sigma.
\ee
Taking a variation of this, we get
\[
0= \int_0^1\left(-2\big( (\dot t + \langle\omega,\dot x\rangle \big) \left( \dot T+ \langle\nabla_X \omega , \dot x  \rangle+\langle\omega,\dot X\rangle \right)  +2(\dot x, \dot X)_h \right) \d \sigma.
\]
We have $\langle\omega,\dot X\rangle = (\d/\d \sigma)\langle\omega,  X\rangle - \langle D_\sigma \omega, X\rangle  $ (where $D_\sigma=\nabla_{\dot x}$), thus
\[
\begin{split}
\langle\nabla_X \omega , \dot x  \rangle+\langle\omega,\dot X\rangle &= \langle\nabla_X \omega , \dot x  \rangle+ \frac{\d}{\d \sigma}\langle\omega,  X\rangle - \langle D_s \omega,  X\rangle \\
&= \frac{\d}{\d \sigma}\langle\omega,  X\rangle + \langle -\d\omega, \dot x\otimes X  \rangle.
\end{split}
\]
After some integration by parts we obtain
\[
0= \int_0^1\left(\frac{\d}{\d \sigma} (\dot t +\langle\omega ,\dot x\rangle \big)  ( T+ \langle \omega,X\rangle)  
+   \left(\dot t +\langle\omega ,\dot x\rangle \right)  \langle \d\omega, \dot x\otimes X  \rangle  -(D_\sigma\dot x,  X)_h  \right)  \d \sigma.
\]
Since this is true for every perturbation $(T,X)$, we get
\be{5.4}
\begin{split}
\frac{\d}{\d \sigma} (\dot t + \langle\omega ,\dot x\rangle  \big) &=0,\\
D_\sigma\dot x^j  &= \left(\dot t + \langle\omega ,\dot x\rangle \right)  {{(\d\omega)}_i}^{j} \dot x^i .
\end{split}
\ee
Therefore, $\dot t + \langle\omega ,\dot x\rangle = k=\text{const}$. This can also be interpreted as $(\partial/\partial t, \dot\gamma)_g=\text{const.}$ for every  geodesic, i.e., the energy of a particle is constant for all stationary observers. We get
\begin{align}  \label{5.5a}
\dot t +\langle\omega ,\dot x\rangle   &=k,\\
D_\sigma\dot x  &= k Y\dot x, \label{5.5b}
\end{align}
where $Y:TN\to TN$ is given by  $(Yu)^i = {{(\d\omega)}_i}^{j}u^i$, in other words, $\langle Yu,v\rangle_h = \langle \d\omega, u\otimes v\rangle$ for every two vector fields $u$ and $v$. 
\begin{figure}[h!] 
  \centering
  \includegraphics[scale=1,page=2]{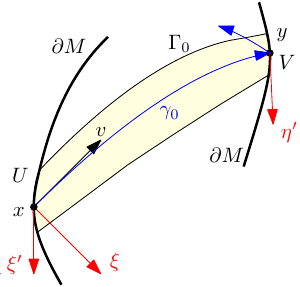}
\caption{\small  A lightlike geodesic and its projection $x(\sigma)$ on $N$.
}
\label{Lorentz_distance_fig2}
\end{figure}
Now we are in the framework of \cite{St-magnetic}, see also Appendix~\ref{app}. Equation \r{5.5b} is  Newton's law of magnetic geodesics with Lorentzian force $kY$. Since the operator $Y$ is anti-symmetric, we have $|\dot x|_h=m=\text{const.}$, i.e., $\sigma$ is proportional to the arc-length parameter on the base $N$, which in general is not (proportional to) the time $t$. The constants $k$ and $m$ are $\sigma$-independent but they may change from geodesic to geodesic, in principle. For $\gamma = (t(\sigma),x(\sigma))$, we have $(\dot\gamma, \dot\gamma)_g = -k^2+m^2$. If $\gamma$ is lightlike and future pointing, we can take $k=m=1$ after replacing $\sigma$ by $k\sigma$.  When $\gamma$  is not lightlike, we can have $k=0$. Then $x(\sigma)$ is a geodesic in $(N,h)$.

Next, integrating \r{5.5a}, we get
\be{5.5c}
t(\sigma)+\int_{\sigma_0}^\sigma \langle\omega,\dot x\rangle-k\sigma=\text{const.}
\ee
In the case $k=1$ which we really need below, this shows that $t$ is actually an action variable, see \r{A1}.

We have thus proved part (a) of Theorem~\ref{thm_mag} below. 
In preparation to formulate part (b), (c), denote by $\mathcal{S}_\text{mag}: BN\to BN$ the scattering relation related to the magnetic system \r{5.5b} with $k=1$ there, which fixes unit speed along the magnetic geodesics. Let $\ell_\text{mag} :BN\to\R_+$ be the travel time, which is also the length of the magnetic geodesic inside $N$ in the metric $h$. The boundary action function $\mathbb A$ is defined in Appendix~\ref{app}. 

Recall that we assumed $\lambda=1$ after Definition~\ref{def_gauge}. Since we will use the theorem below for $\mathcal{S}^\sharp$, which is independent of $\lambda$, this is enough for our purposes. The theorem is easier to formulate for $\mathcal{S}$ however. 

\begin{theorem}\label{thm_mag}
Let $(N,h)$ be a compact Riemannian manifold with a strictly convex boundary, let $\omega$ be an one-form on $M$, and let  $M=\R\times N$ be equipped with the stationary Lorentzian metric
\be{thm_mag1}
g = -   (\d t+\langle w,\d x\rangle)^2 + h(x) .
\ee
Then we have the following.

(a) The lightlike geodesics $(t(\sigma),x(\sigma))$ solve \r{5.5a}, \r{5.5b} with $k$ independent of $\sigma$ reflecting the freedom of affine changes of $\sigma$. Future propagation corresponds to $k>0$.  For the magnetic geodesics $x(\sigma)$ we have $|\dot x|_h=|k|$.

(b) With the notation in \r{5.3xxx}, if 
\be{5.6}
\mathcal{S} \left(t,x,[v_t', v_x']  \right) = \left(s,y,[w_t', w_x']\right),
\ee
with $v'$  
normalized so that $v_t'=1$, then $w_t'=1$ as well, and 
\be{5.7}
\mathcal{S}_\text{\rm mag}(x,v'_x) = (y,w_x'), \quad \ell_\text{\rm mag}(x,v_x') = s-t+ \int_{\gamma_{x,v'}}\omega,
\ee
and in particular,
\be{5.8}
\mathbb{A}(x,y) =s-t \quad \text{with $ y= \exp_x^\text{\rm mag} v$},
\ee
where $v\in T_xM$ is the incoming lightlike vector with projection $v'$, and $\exp_x^\text{\rm mag}$ is the magnetic exponential map.

(c) Each one of the three quantities determines the other two: $\mathcal S$, $\mathcal{S}_\text{\rm mag}$, and $\mathbb A$. 
\end{theorem}

\subsection{Reduction of the scattering relation to a magnetic one}

\begin{proof}[End of the proof of Theorem~\ref{thm_mag}]
 
 By \r{5.5a}, \r{5.5b}, for the geodesic $\gamma_{x,v'}= (x(\sigma), t(\sigma))$ issued from $(x,v)$ with projection $v'$, normalized as in the theorem, we have $k=1= |\dot x|_h$ in \r{5.5a}, \r{5.5b}.  Projecting on $N$, we get the first identity in \r{5.7}. For the second one, we refer to \r{5.5c}: with $k=1$ there and an initial condition $t(\sigma)=t$ (the $t$ on the right is the initial moment in \r{5.6} and the one on the left is the zeroth component of $\gamma_{x,v'}$), we get that $s-t$ is equal to the action along the ray, which proves the rest of the (b) statement. 

The proof of (c) is similar to that in the Riemannian case. That $\mathcal{S}_\text{mag}$ and $\mathbb A$ determine each other is proved in \cite{St-magnetic}, and the proof is similar to the arguments in section~\ref{sec_R}. 
Knowing $\mathcal S$, we recover $\mathcal{S}_\text{mag}$ by part (b). Finally, knowing one of $\mathcal{S}_\text{mag}$ and $\mathbb A$, we recover the other as well, and then we recover $\mathcal S$ by \r{5.7} and \r{5.8}. 
\end{proof}

\subsection{Rigidity results} 
We formulate some rigidity results as a consequence of the equivalence between the rigidity problems for stationary Lorentzian metrics and magnetic systems established in Theorem~\ref{thm_mag}; and from the magnetic rigidity results in \cite{St-magnetic}, see the appendix. We call $(M,g)$ simple, if the projected magnetic system is simple, see the appendix. In the next statements, $g$ is of the form \r{5.2}, i.e., $g=g_{\lambda, \omega,h}$. Given another such metric $\hat g$, all associated quantities are decorated with a hat over them.

\begin{theorem}
Let $(M,g)$ and $( M, \hat g)$ be simple and stationary. Then $\mathcal S^\sharp =\hat{\mathcal{S}}^\sharp$ implies that 
$g$ and $\hat g$  are gauge equivalent if and only if 
$\mathbb{A}= \hat{\mathbb{A}}$
implies that $(h,\omega )$  and $(\hat h,\hat \omega )$  are magnetically gauge equivalent.  
In particular,  the simple magnetic system $(N,h,\omega)$ is boundary rigid if and only if the stationary Lorentzian manifold $(M,g_{\lambda,\omega,h})$ is boundary rigid. 
\end{theorem}

\begin{corollary}[Rigidity in $1+2$ dimensions]
Simple stationary manifolds $(M,g)$ of dimension $\dim M=1+2$ are lens rigid. 
\end{corollary}

\begin{corollary}[Rigidity in a given conformal class]
Let 
$(M,g)$ and $(M,\hat g)$ 
be two simple stationary Lorentzian systems so that $\hat h = \mu(x) h$ with some $\mu>0$. If $\mathcal{S}^\sharp= \hat{\mathcal{S}}^\sharp$, then $\mu=1$ and $\hat \omega = \omega+\d\phi(x)$ with some $\phi$ vanishing on $\partial N$.
\end{corollary}

For the definition of the class $\mathcal{G}^k$ used in next theorem, we refer to Definition~\ref{def_G}. 
\begin{corollary}[Generic local rigidity]
There exists $k\ge k_0$ so that for every $(h_0,\omega_0)\in \mathcal{G}^k$, there exists $\eps>0$ such that for every two simple stationary metrics $g=g_{\lambda,\omega,h}$ and $\hat g= g_{\hat\lambda, \hat\omega, \hat h}$ 
for each of which  $(h,\alpha)$, $(\hat h, \hat \alpha)$ is an $\eps$ close to $(h_0,\alpha_0)$ in $C^k(N)$, we have the following:
\[
\mathcal{S}^\sharp= \hat{\mathcal{S}}^\sharp
\]
implies that $\hat g$ and $g$   are gauge equivalent.
\end{corollary}

\subsection{The defining function of $\Sigma$ for stationary metrics} 
Although we did not need to resort to a defining function $r$ in the stationary case, it would be interesting to see how the general theory developed in section~\ref{sec2} applies here. 
We will use the energy, see \r{7} as a definition of $r$. For $(x,y)$ close to some $(x_0,y_0)\in \Sigma$, we have $|r(x,y)|\ll1$, on the other hand, $r=-k^2+m^2$, where $k$ and $m$ are as in \r{5.5a}, \r{5.5b}, and the remarks after it, related to the geodesic $\gamma_{[x,y]}(\sigma)$ parameterized by $\sigma\in [0,1]$. Since $|\dot x|_h=m$, we have $m=\ell_{x,y}$. Integrating   \r{5.5a} along $\gamma$, we get
\[
k=(s-t) +\int_{\gamma_{[x,y]}} \omega .
\]
Therefore,
\[
r(t,x,s,y) = \ell_{x,y}^2 - \Big( (s-t) +\int_{\gamma_{[x,y]}} \omega \Big)^2. 
\]
Up to an elliptic factor, we can replace $r$ by the defining function
\[
r_1(t,x,s,y) = -  (s-t)+ \ell_{x,y} -\int_{\gamma_{[x,y]}} \omega =   t-s + \mathbb{A}(x,y). 
\]
This is a direct generalization of what we got in Example~\ref{ex1} for $g=-\d t^2+h(x,\d x)$. 

A linearization of $r_1$, by \cite[Lemma~3.1]{St-magnetic}, is given by
\be{5.9}
\frac12 \int \langle \delta h, \dot\gamma\oplus \dot\gamma\rangle -\int_\gamma\delta\omega
\ee
with $\gamma$ parameterized by an arc-length parameter, i.e., $|\dot\gamma|_h=1$. 
We will compare it with the linearization \r{thm1_1} in Theorem~\ref{thm1}(c). By \r{thm1_1}, we should get
\be{5.9a}
\int_0^1 \left\langle f, \dot \gamma_{[x,y]} (\sigma)\otimes  \dot \gamma_{[x,y]} (\sigma)\right\rangle\,\d \sigma\quad \text{on $\Sigma= \{r(x,y)=0\}$},
\ee
where $f=\delta g$ with $g$ as in \r{5.3x}. Therefore, dropping the subscript $[x,y]$, we get
\be{5.10}
\left\langle f, \dot \gamma  (\sigma)\otimes  \dot \gamma  (\sigma)\right\rangle 
= -2(\dot\gamma_t + \langle \omega, \dot\gamma_x\rangle)\langle \delta\omega, \dot\gamma_x\rangle + \langle \delta h, \dot\gamma_x\otimes \dot\gamma_x\rangle,
\ee
where the (new) subscripts $t$ and $x$ denote the time and the spatial components of $\gamma$, respectively. 
By \r{5.5a}, \r{5.5b} or \r{5.3xx}, we have $\dot\gamma_t + \langle \omega, \dot\gamma_x\rangle= |\dot\gamma_x|$ for $\gamma$ future pointing. In \r{5.10}, $\gamma(\sigma)$ is parameterized by $\sigma\in [0,1]$, thus $|\dot\gamma_{x,y]}|=\ell_{x,y}$. The rescaling $\tilde \sigma = \ell_{x,y}\sigma$ makes $\tilde\sigma$ an arc-length parameter. Doing this in \r{5.10}, we see that it equals $2\ell_{x,y}^2$ times \r{5.9}. Therefore, the linearization \r{5.9} for the magnetic rigidity problem obtained in \cite{St-magnetic} coincides with the linearization \r{5.9a}, predicted by Theorem~\ref{thm1}(c).

\appendix
\section{Some facts about magnetic systems} \label{app}
We recall some notions and results in \cite{St-magnetic}. On a Riemannian manifold $(N,h)$, we are given a closed two-form $\Omega$, which in our case would be $\Omega=\d\omega$. We define $Y:TN\to TN$ by $\langle \Omega, u\otimes v\rangle = (Yu,v)_h$. Then we consider the Newton-like equation
\[
D_s \dot\gamma = Y\dot\gamma.
\]
The solution curves $\gamma(\sigma)$ are called magnetic geodesics. An easy calculation yields $|\dot \gamma|=\text{const.}$ along each geodesic. Choosing different values of that constant generates different curves; and we fix $|\dot \gamma|=1$. Time   is not reversible along $\gamma$ unless $\Omega=0$. 

We call $N$ simple with respect to $h$ and $\Omega$ of the magnetic exponential map at every point $x$ is a diffeomorphism to $N$ from its pre-image, and if $\partial N$ is strictly convex with respect to the magnetic flow in either direction. Then $\Omega=\d\omega$ with some one-form $\omega$. 
We view $(N,h,\omega)$ as a magnetic system. 

For every pair $(x,y)\in N\times N$, one defines the action $\mathbb{A}(x,y)$ by
\be{A1}
\mathbb{A}(x,y) = \ell_{x,y}-\int_{\gamma_{[x,y]}}\omega,
\ee
where $\gamma_{[x,y]}$ is the unique (by simplicity) unit speed magnetic geodesic from $x$ to $y$, and $\ell_{x,y}$ is the travel time. The action minimizes a certain time-free action functional. Restricted to $\partial N\times\partial N$, $\mathbb{A}$ is called the boundary action function. It plays the role, and generalizes  of the boundary distance function when $\omega=0$. 
Two magnetic systems $(N,h,\omega)$  and $(N,\hat h,\hat \omega)$ are called gauge equivalent if there exists a diffeomorphism $\psi$ on $N$ fixing $\partial N$ pointwise, and a function $\phi$ vanishing on $\partial N$, so that $\hat h = \psi^*h$, and $\hat \omega = \psi^*\omega+\d\phi$. Gauge equivalent magnetic systems have the same boundary action functions. 

One defines the scattering relation $\mathcal{S}_\text{mag}$, and the travel time $\ell_\text{mag}$ in the same was as we did in the Riemannian case. With the notation $\mathcal{S}_\text{mag}(x,v')=(y,w')$, we have the following generalization of \r{1}:
\be{A2}
v'=-\d'_x\mathbb{A}(x,y)+\omega'(x), \quad w'=\d'_y\mathbb{A}(x,y)+\omega'(y),
\ee
where, as before, primes denote tangential projections, and in particular, $\omega'$ is $\omega$ restricted to $T\partial N$. 

\subsection{Linearization} A linearization of $\mathbb{A}(x,y)$ is the following X-ray transform
\be{A3}
I[ f, \beta](\gamma) =   \int \langle  f, \dot\gamma\otimes \dot\gamma\rangle+ \int \beta,
\ee
see \cite[Lemma~3.1]{St-magnetic}, where $f$ is a symmetric two-cotensor field, and $\beta$ is an one-form, which play roles of perturbations of the background metric $h$ (multiplied by $1/2$), and form $-\omega$.  Lift $Y$ to the cotangent bundle by $(Y\omega)_i = -{Y_i}^j \omega_j$, which corresponds to the isomorphism between $TN$ and $T^*N$. 

Based on the obvious gauge invariance of the nonlinear problem, after linearization, 
we expect $I[f,\beta]=0$ to imply $h=\d^s v$ with $v=0$ on $\partial N$, and  $\beta= \d\phi - Yv$ with $\phi$ a function vanishing on $\partial N$. We called this property s-injectivity. 
We proved s-injectivity for simple magnetic systems in the following cases

\begin{itemize}

\item[(i)] with an explicit bound of the curvature, following the energy method going back to Mukhometov, Romanov, Pestov and Sharafutdinov; 

\item[(ii)] in a given conformal class, 

\item[(iii)] for analytic ones using analytic microlocal analysis,  

\item[(iv)] locally, near generic ones using the analytic result. 
\end{itemize}

To describe the latter, we need the following definition.

\begin{definition}\label{def_G}
We define $\mathcal{G}^k$ to be the set of all simple $C^k$ pairs $(h,\omega)$ on $N$ with an s-injective magnetic ray transform $I_{h,\omega}$.
\end{definition}

We can define a $C^k$ topology on $N$ (independent of a metric $h$ which we eventually impose) by taking and fixing a finite atlas of local maps.

\begin{theorem}[\!\!{\cite[Theorem~4.11]{St-magnetic}}]\label{thm_gen_m}
For some $k_0>0$, for every $k\ge k_0$,  the set $\mathcal{G}^k$  is open and dense in the set of all $C^k$ pairs $(h,\beta)$ and contains all real analytic simple pairs. 
\end{theorem}

The magnetic ray transform is elliptic on the complement of the potential pairs, which allows for a stability estimate, which in turn allows to apply this to the nonlinear problem. 

\subsection{Rigidity results} We sketch the rigidity results about simple magnetic system obtained in \cite{St-magnetic}. We proved boundary determination of the whole jet of $h$ and $\omega$, up to the gauge, first. Next, we showed the following results. 

\begin{itemize}
\item[(i)] Two-dimensional (simple) magnetic systems are boundary rigid. This was derived generalizing the Riemannian result by Pestov and Uhlmann \cite{PestovU}, without a linearization. 

\item[(ii)] If $\hat g = \mu g$ with $\mu>0$ a function, then equality of the lens data implies $\mu=1$ and $\hat\omega$ is gauge equivalent to $\omega$. 

\item[(iii)] Real analytic simple magnetic systems with the same lens data are gauge equivalent. This follows from a boundary determination of the jets of $h$ and $\omega$, and then by analytic continuation. 

\item[(iv)] Generic local rigidity near simple magnetic systems with s-injective linearizations, following \cite{SU-JAMS}. 
  
\end{itemize}

Recovery of a conformal factor and  $\d\omega$ from local data near a strictly convex boundary point, and a global result under a foliation condition was proved in \cite{Hanming-18}. This requires knowledge of the action $\mathbb{A}$ \textit{and} $\ell$. 

We formulate (iv) in the following.

\begin{theorem}[\!{\cite[Theorem~6.5]{St-magnetic}}]\label{thm_gen_m_nl} 
There exists $k\ge k_0$ so that for every $(h_0,\omega_0)\in \mathcal{G}^k$, there exists $\eps>0$ such that for every two magnetic systems $(h,\omega)$, $(\hat h, \hat \omega)$, each of which is an $\eps$ close to $(h_0,\omega_0)$ in $C^k(N)$, we have the following:
\[
\hat{\mathbb{A}}= \mathbb {A}\quad \text{on $\partial N\times\partial N$}
\]
implies that $(\hat h, \hat\omega)$ and $(h,\omega)$   are gauge equivalent. 
\end{theorem}


\begin{thebibliography}{10}

\bibitem{Alias_2009}
L.~J. Alias, A.~Hurtado, and V.~Palmer.
\newblock Comparison theory of {L}orentzian distance with applications to
  spacelike hypersurfaces.
\newblock {\em AIP Conference Proceedings}, 1122(1):91--98, 2009.

\bibitem{Arnold_Mech}
V.~I. Arnol'd.
\newblock {\em Mathematical methods of classical mechanics}, volume~60 of {\em
  Graduate Texts in Mathematics}.
\newblock Springer-Verlag, New York, second edition, 1989.
\newblock Translated from the Russian by K. Vogtmann and A. Weinstein.

\bibitem{Yernat-Hanming-15}
Y.~M. Assylbekov and H.~Zhou.
\newblock Boundary and scattering rigidity problems in the presence of a
  magnetic field and a potential.
\newblock {\em Inverse Probl. Imaging}, 9(4):935--950, 2015.

\bibitem{bernstein1980conditions}
I.~Bernstein and M.~Gerver.
\newblock Conditions on distinguishability of metrics by hodographs.
\newblock {\em Methods and Algorithms of Interpretation of Seismological
  Information}, pages 50--73, 1980.

\bibitem{Burago-Ivanov}
D.~Burago and S.~Ivanov.
\newblock Boundary rigidity and filling volume minimality of metrics close to a
  flat one.
\newblock {\em Ann. of Math. (2)}, 171(2):1183--1211, 2010.

\bibitem{Croke_scatteringrigidity}
C.~Croke.
\newblock Scattering rigidity with trapped geodesics.
\newblock {\em Ergodic Theory Dynam. Systems}, 34(3):826--836, 2014.

\bibitem{Croke90}
C.~B. Croke.
\newblock Rigidity for surfaces of nonpositive curvature.
\newblock {\em Comment. Math. Helv.}, 65(1):150--169, 1990.

\bibitem{Croke04}
C.~B. Croke.
\newblock Rigidity theorems in {R}iemannian geometry.
\newblock In {\em Geometric methods in inverse problems and PDE control},
  volume 137 of {\em IMA Vol. Math. Appl.}, pages 47--72. Springer, New York,
  2004.

\bibitem{CDS}
C.~B. Croke, N.~S. Dairbekov, and V.~A. Sharafutdinov.
\newblock Local boundary rigidity of a compact {R}iemannian manifold with
  curvature bounded above.
\newblock {\em Trans. Amer. Math. Soc.}, 352(9):3937--3956, 2000.

\bibitem{St-magnetic}
N.~S. Dairbekov, G.~P. Paternain, P.~Stefanov, and G.~Uhlmann.
\newblock The boundary rigidity problem in the presence of a magnetic field.
\newblock {\em Adv. Math.}, 216(2):535--609, 2007.

\bibitem{Denisjuk_06}
A.~Denisjuk.
\newblock Inversion of the x-ray transform for 3{D} symmetric tensor fields
  with sources on a curve.
\newblock {\em Inverse Problems}, 22(2):399--411, 2006.

\bibitem{Feizmohammadi2019}
A.~Feizmohammadi, J.~Ilmavirta, Y.~Kian, and L.~Oksanen.
\newblock Recovery of time dependent coefficients from boundary data for
  hyperbolic equations.
\newblock {\em Preprint arXiv:1901.04211}.

\bibitem{Lauri-Light-21}
A.~Feizmohammadi, J.~Ilmavirta, and L.~Oksanen.
\newblock The light ray transform in stationary and static {L}orentzian
  geometries.
\newblock {\em J. Geom. Anal.}, 31(4):3656--3682, 2021.

\bibitem{Flores2002}
J.~L. Flores and M.~S\'{a}nchez.
\newblock Geodesics in stationary spacetimes. {A}pplication to {K}err
  spacetime.
\newblock {\em Int. J. Theor. Phys. Group Theory Nonlinear Opt.},
  8(3):319--336, 2002.

\bibitem{Germinario_07}
A.~Germinario.
\newblock Geodesics in stationary spacetimes and classical {L}agrangian
  systems.
\newblock {\em J. Differential Equations}, 232(1):253--276, 2007.

\bibitem{Gromov}
M.~Gromov.
\newblock Filling {R}iemannian manifolds.
\newblock {\em J. Differential Geom.}, 18(1):1--147, 1983.

\bibitem{Colin14}
C.~Guillarmou.
\newblock Lens rigidity for manifolds with hyperbolic trapped set.
\newblock {\em arXiv:1412.1760}, 2014.

\bibitem{Herglotz}
G.~Herglotz.
\newblock {\"U}ber die {E}lastizitaet der {E}rde bei {B}eruecksichtigung ihrer
  variablen {D}ichte.
\newblock {\em Zeitschr. f\"ur Math. Phys.}, 52:275--299, 1905.

\bibitem{LOSU-strings}
M.~Lassas, L.~Oksanen, P.~Stefanov, and G.~Uhlmann.
\newblock On the {I}nverse {P}roblem of {F}inding {C}osmic {S}trings and
  {O}ther {T}opological {D}efects.
\newblock {\em Comm. Math. Phys.}, 357(2):569--595, 2018.

\bibitem{LOSU-Light_Ray}
M.~Lassas, L.~Oksanen, P.~Stefanov, and G.~Uhlmann.
\newblock The light ray transform on {L}orentzian manifolds.
\newblock {\em Comm. Math. Phys.}, 377(2):1349--1379, 2020.

\bibitem{LassasSU}
M.~Lassas, V.~Sharafutdinov, and G.~Uhlmann.
\newblock Semiglobal boundary rigidity for {R}iemannian metrics.
\newblock {\em Math. Ann.}, 325(4):767--793, 2003.

\bibitem{Michel}
R.~Michel.
\newblock Sur la rigidit\'e impos\'ee par la longueur des g\'eod\'esiques.
\newblock {\em Invent. Math.}, 65(1):71--83, 1981/82.

\bibitem{Monard-2016}
F.~Monard.
\newblock Efficient tensor tomography in fan-beam coordinates.
\newblock {\em Inverse Probl. Imaging}, 10(2):433--459, 2016.

\bibitem{Mu2}
R.~G. Muhometov.
\newblock On a problem of reconstructing {R}iemannian metrics.
\newblock {\em Sibirsk. Mat. Zh.}, 22(3):119--135, 237, 1981.

\bibitem{MuRo}
R.~G. Muhometov and V.~G. Romanov.
\newblock On the problem of finding an isotropic {R}iemannian metric in an
  {$n$}-dimensional space.
\newblock {\em Dokl. Akad. Nauk SSSR}, 243(1):41--44, 1978.

\bibitem{ONeill-book}
B.~O'Neill.
\newblock {\em Semi-{R}iemannian {G}eometry With {A}pplications to
  {R}elativity}, volume 103 of {\em Pure and Applied Mathematics}.
\newblock Academic Press, Inc. [Harcourt Brace Jovanovich, Publishers], New
  York, 1983.

\bibitem{Otal}
J.-P. Otal.
\newblock Sur les longueurs des g\'eod\'esiques d'une m\'etrique \`a courbure
  n\'egative dans le disque.
\newblock {\em Comment. Math. Helv.}, 65(2):334--347, 1990.

\bibitem{PestovU}
L.~Pestov and G.~Uhlmann.
\newblock Two dimensional compact simple {R}iemannian manifolds are boundary
  distance rigid.
\newblock {\em Ann. of Math. (2)}, 161(2):1093--1110, 2005.

\bibitem{Siamak2016}
S.~RabieniaHaratbar.
\newblock Support theorem for the light-ray transform of vector fields on
  {M}inkowski spaces.
\newblock {\em Inverse Probl. Imaging}, 12(2):293--314, 2018.

\bibitem{Sh-book}
V.~Sharafutdinov.
\newblock {\em Integral geometry of tensor fields}.
\newblock Inverse and Ill-posed Problems Series. VSP, Utrecht, 1994.

\bibitem{MR1004174}
P.~Stefanov.
\newblock Uniqueness of the multi-dimensional inverse scattering problem for
  time dependent potentials.
\newblock {\em Math. Z.}, 201(4):541--559, 1989.

\bibitem{S-Serdica}
P.~Stefanov.
\newblock Microlocal approach to tensor tomography and boundary and lens
  rigidity.
\newblock {\em Serdica Mathematical Journal}, 34(1):67--112, 2008.

\bibitem{S-AIP}
P.~Stefanov.
\newblock A sharp stability estimate in tensor tomography.
\newblock {\em Journal of Physics: Conference Series}, 124(1):012007, 2008.

\bibitem{S-support2014}
P.~Stefanov.
\newblock Support theorems for the light ray transform on analytic {L}orentzian
  manifolds.
\newblock {\em Proc. Amer. Math. Soc.}, 145(3):1259--1274, 2017.

\bibitem{SU-MRL}
P.~Stefanov and G.~Uhlmann.
\newblock Rigidity for metrics with the same lengths of geodesics.
\newblock {\em Math. Res. Lett.}, 5(1-2):83--96, 1998.

\bibitem{SU-Duke}
P.~Stefanov and G.~Uhlmann.
\newblock Stability estimates for the {X}-ray transform of tensor fields and
  boundary rigidity.
\newblock {\em Duke Math. J.}, 123(3):445--467, 2004.

\bibitem{SU-JAMS}
P.~Stefanov and G.~Uhlmann.
\newblock Boundary rigidity and stability for generic simple metrics.
\newblock {\em J. Amer. Math. Soc.}, 18(4):975--1003, 2005.

\bibitem{SU-Kawai}
P.~Stefanov and G.~Uhlmann.
\newblock Boundary and lens rigidity, tensor tomography and analytic microlocal
  analysis.
\newblock In {\em Algebraic analysis of differential equations from microlocal
  analysis to exponential asymptotics}, pages 275--293. Springer, Tokyo, 2008.

\bibitem{SU-AJM}
P.~Stefanov and G.~Uhlmann.
\newblock Integral geometry on tensor fields on a class of non-simple
  {R}iemannian manifolds.
\newblock {\em Amer. J. Math.}, 130(1):239--268, 2008.

\bibitem{SU-JFA09}
P.~Stefanov and G.~Uhlmann.
\newblock Linearizing non-linear inverse problems and an application to inverse
  backscattering.
\newblock {\em J. Funct. Anal.}, 256(9):2842--2866, 2009.

\bibitem{SU-lens}
P.~Stefanov and G.~Uhlmann.
\newblock Local lens rigidity with incomplete data for a class of non-simple
  {R}iemannian manifolds.
\newblock {\em J. Differential Geom.}, 82(2):383--409, 2009.

\bibitem{SUV_localrigidity}
P.~Stefanov, G.~Uhlmann, and A.~Vasy.
\newblock Boundary rigidity with partial data.
\newblock {\em J. Amer. Math. Soc.}, 29(2):299--332, 2016.

\bibitem{SUV-tensors}
P.~Stefanov, G.~Uhlmann, and A.~Vasy.
\newblock Inverting the local geodesic {X}-ray transform on tensors.
\newblock {\em Journal d'Analyse Math\'ematique}, 136(1):151--208, 2018.

\bibitem{SUV_anisotropic}
P.~Stefanov, G.~Uhlmann, and A.~Vasy.
\newblock Local and global boundary rigidity and the geodesic {X}-ray transform
  in the normal gauge.
\newblock {\em Ann. of Math. (2)}, 194(1):1--95, 2021.

\bibitem{St-Yang-DN}
P.~Stefanov and Y.~Yang.
\newblock The inverse problem for the {D}irichlet-to-{N}eumann map on
  {L}orentzian manifolds.
\newblock {\em Anal. PDE}, 11(6):1381--1414, 2018.

\bibitem{StrohmaierZ}
A.~Strohmaier and S.~Zelditch.
\newblock Spectral asymptotics on stationary space-times.
\newblock {\em Rev. Math. Phys.}, 33(1):Paper No. 2060007, 14, 2021.

\bibitem{UV:local}
G.~Uhlmann and A.~Vasy.
\newblock The inverse problem for the local geodesic ray transform.
\newblock {\em Invent. Math.}, 205(1):83--120, 2016.

\bibitem{UhlmannYangZhou}
G.~Uhlmann, Y.~Yang, and H.~Zhou.
\newblock Travel time tomography in stationary spacetimes.
\newblock {\em J. Geom. Anal.}, 31(10):9573--9596, 2021.

\bibitem{Vasy-Wang-21}
A.~Vasy and Y.~Wang.
\newblock On the light ray transform of wave equation solutions.
\newblock {\em Comm. Math. Phys.}, 384(1):503--532, 2021.

\bibitem{wang2017parametrices}
Y.~Wang.
\newblock Parametrices for the light ray transform on {M}inkowski spacetime.
\newblock {\em Inverse Probl. Imaging}, 12(1):229--237, 2018.

\bibitem{WZ}
E.~Wiechert and K.~Zoeppritz.
\newblock {\"U}ber {E}rdbebenwellen.
\newblock {\em Nachr. Koenigl. Geselschaft Wiss. G\"ottingen}, 4:415--549,
  1907.

\bibitem{Hanming-18}
H.~Zhou.
\newblock Lens rigidity with partial data in the presence of a magnetic field.
\newblock {\em Inverse Probl. Imaging}, 12(6):1365--1387, 2018.

\end{thebibliography}

\end{document}